\documentclass[a4paper,12pt]{paper}

\usepackage{graphicx}
\usepackage{amsmath,amsthm}
\usepackage{amsfonts,amssymb}
\usepackage{amscd}
\usepackage{color}
\usepackage[cp1251]{inputenc}
\usepackage[matrix,arrow,curve]{xy}
\usepackage{tikz}
\usetikzlibrary{arrows}
\usepackage{mathtools}

\newtheorem{lemma}{Lemma}[section]
\newtheorem{prop}[lemma]{Proposition}

\newtheorem*{Mtheorem}{Main Theorem}
\newtheorem{coro}[lemma]{Corollary}
\newtheorem{rema}[lemma]{{\rm\bf Remark}}
\newtheorem{Def}[lemma]{Definition}

\newtheorem{ex}[lemma]{Example}

\newcommand{\kk}{{\bf k}}

\newcommand{\la}{\langle}
\newcommand{\ra}{\rangle}

\renewcommand{\le}{\leqslant}
\renewcommand{\ge}{\geqslant}

\newenvironment{Ex}{\begin{ex}
\rm }{\end{ex}}

\numberwithin{equation}{section}
\begin{document}

\title{$n$-ary algebras of the first level.}

\author{Yury Volkov}
%\date{}
\maketitle
\begin{abstract}
$n$-ary algebras of the first degeneration level are described. A detailed classification is given in the cases $n=2,3$.
\end{abstract}


       \vspace{0.3cm}


       \vspace{0.3cm}

\section{Introduction}

Algebras considered in this paper are by definition linear spaces with $n$-ary $n$-linear operation that does not have to satisfy any other properties.
The main object considered in this paper is the  degeneration of algebras.
Roughly speaking, the algebra $A$ degenerates to the algebra $B$ if there is a family of algebra structures parameterized by an element of the ground field such that infinitely many structures in the family represent $A$ and there exists a structure belonging to this family representing $B$. The notion of a degeneration is closely related to the notions of a contraction and a deformation.

Degenerations of binary algebras were considered for a long time in many works (see, for example, \cite{BC99, kpv17, S90}). On the other hand, degenerations of $n$-ary algebras is a relatively novel area. Papers \cite{deaz, mak16} concern this topic but do not give any complete result. Recently in the paper \cite{kpv20} the degenerations in the variety of $(n+1)$-dimensional Filippov ($n$-Lie) algebras were described. The main tools for studying degenerations in the current paper will be taken from \cite{kpv20}. In fact, these tools are generalizations of methods for binary algebras given in \cite{BC99,S90}.

The notion of the level of a binary algebra was introduced in \cite{gorb91}. The algebra under consideration is an algebra of level $m$ if the maximal length of a chain of non-trivial degenerations starting at it equals $m$.
Roughly speaking, the level estimates the complexity of the multiplication of  the given algebra. For example, the unique algebra of the level zero is the algebra with zero multiplication. 

Let us give a short overview of results on levels of binary algebras.
Anticommutative algebras of the first level were classified correctly in \cite{gorb91} and all algebras of the first level were classified in \cite{khud13}. The full classification of algebras of the second level was done in \cite{kpv19} while some particular results appeared before. Nilpotent anticommutative algebras of the levels from three to five were classified in \cite{Volk20}.
In \cite{gorb93} the author introduced the notion of the infinite level that can be expressed in terms of the usual level. The consideration of infinite level for nilpotent algebras allows to exclude some marginal cases from the consideration while for non nilpotent algebras the replacement of the usual level by the infinite level creates a loss of a significant part of classification.

This paper is an initiation of the level classification for $n$-ary algebras. Namely, we will describe $n$-ary algebras of the first level. This classification is in some sense similar to the classification of binary algebras of the first level, but occurs to be much more complicated. As it was mentioned above, we will automatically get the classification of algebras of the first infinite level too. We will also clarify this description in the cases $n=2,3$.

\section{Background on degenerations}

In this section we introduce some notation and recall some well known definitions and results about degenerations that we will need in this work.

All vector spaces in this paper are over some fixed algebraically closed field ${\bf k}$ and we write simply $dim$, $Hom$ and $\otimes$ instead of $dim_{{\bf k}}$, $Hom_{{\bf k}}$ and $\otimes_{{\bf k}}$.
We will write also $A^{\otimes n}$ instead of $\underbrace{{A}\otimes \ldots \otimes {A}}_{n\mbox{\scriptsize { times}}}$.
During this paper we fix some integer $n\ge 2$ and mean by an algebra a vector space with an $n$-ary $n$-linear operation called multiplication, i.e. a space $A$ with a fixed linear map from $A^{\otimes n}$ to $A$.
For an algebra $A$ and $a_1,\dots,a_n\in A$ we will denote the result of the application of multiplication to the $n$-tuple $(a_1,\dots,a_n)$ by $[a_1,\dots,a_n]$.
If $V$ is a linear space and $S$ is a subset of $V$, then we denote by $\la S\ra$ the subspace of $V$ generated by $S$.

Let $V$ be a fixed $m$-dimensional space. Then the {\it set of $m$-dimensional algebra structures} on $V$ is $\mathcal{A}_m=Hom(V^{\otimes m}, V)\cong (V^*)^{\otimes m}\otimes V$. Any $m$-dimensional algebra can be represented by some element of $\mathcal{A}_m$. Two algebras are isomorphic if and only if they can be represented by the same structure.
The set $\mathcal{A}_m$ has a structure of the affine variety ${\bf k}^{m^{n+1}}$.
There is a natural action of the group $GL(V)$ on $\mathcal{A}_m$ defined by the equality $ (g * \mu )(x_1\otimes\dots\otimes x_n) = g\mu(g^{-1}x_1\otimes\dots\otimes g^{-1}x_n)$ for $x_1,\dots,x_n\in V$, $\mu\in \mathcal{A}_m$ and $g\in GL(V)$.
Two structures represent the same algebra if and only if they belong to the same orbit. By $\kk^m$ we will denote the $m$-dimensional algebra with zero multiplication and the structure representing it.
For brevity, we will write $\mu(x_1,\dots,x_n)$ or, if the structure $\mu$ is clear from the context, even $[x_1,\dots,x_n]$ instead of $\mu(x_1\otimes\dots\otimes x_n)$ for $x_1,\dots,x_n\in V$.

Let $A$ and $B$ be $m$-dimensional algebras. Suppose that $\mu,\lambda\in \mathcal{A}_m$ represent $A$ and $B$ respectively. We say that $A$ {\it degenerates} to $B$ and write $A\to B$ if $\lambda$ belongs to $\overline{O(\mu)}$. Here, as usually, $O(X)$ denotes the orbit of $X$ and $\overline{X}$ denotes the Zariski closure of $X$ . We also write $A\not\to B$ if $\lambda\not\in\overline{O(\mu)}$. We say that the degeneration $A\to B$ is {\it trivial} if $A\not\cong B$. We will write $A\xrightarrow{\not\cong} B$ to emphasize that the degeneration $A\to B$ is not trivial.

Whenever an $m$-dimensional space named $V$ appears in this paper, we assume that there is some fixed basis $e_1,\dots, e_m$ of $V$. In this case, for $\mu\in\mathcal{A}_n$, we denote by $\mu_{i_1,\dots,i_n}^j$ ($1\le i_1,\dots,i_n,j\le m$) the structure constants of $\mu$ in this fixed basis, i.e. elements of $\bf k$ such that $\mu(e_{i_1},\dots,e_{i_n})=\sum\limits_{j=1}^m\mu_{i_1,\dots,i_n}^je_j$. To prove degenerations and nondegenerations we will use the same technique that has been already used in \cite{kpv17,kpv19,S90} for binary algebras and was generalized in \cite{kpv20} to the $n$-ary case. In particular, we will be free to use \cite[Lemma 1]{kpv20} and facts that easily follow from it.
Let us recall this lemma. If $A\to B$, $\mu\in\mathcal{A}_m$ represents $A$ and there is a closed subset $\mathcal{R}\subset\mathcal{A}_m$ invariant under lower triangular transformations of the basis $e_1,\dots,e_m$ such that $\mu\in\mathcal{R}$, then there is a structure $\lambda\in\mathcal{R}$ representing $B$. Invariance under lower triangular transformations of the basis $e_1,\dots,e_m$ means that if $\omega\in\mathcal{R}$ and $g\in GL(V)$ has a lower triangular matrix in the basis $e_1,\dots,e_m$, then $g*\omega\in\mathcal{R}$ (see \cite{kpv20} for a more detailed discussion).

To prove degenerations, we will use the technique of contractions. Namely, let $\mu,\lambda\in \mathcal{A}_m$ represent $A$ and $B$ respectively. Suppose that there are some elements $E_i^t\in V$ ($1\le i\le m$, $t\in{\bf k}^*$) such that $E^t=(E_1^t,\dots,E_m^t)$ is a basis of $V$ for any $t\in{\bf k}^*$ and the structure constants of $\mu$ in this basis are $\mu_{i_1,\dots,i_n}^j(t)$ for some polynomials $\mu_{i_1,\dots,i_n}^j(t)\in{\bf k}[t]$. If $\mu_{i_1,\dots,i_n}^j(0)=\lambda_{i_1,\dots,i_n}^j$ for all $1\le i_1,\dots,i_n,j\le m$, then $A\to B$. To emphasize that the {\it parameterized basis} $E^t=(E_1^t,\dots,E_m^t)$ ($t\in{\bf k}^*$) gives a degeneration between algebras represented by the structures $\mu$ and $\lambda$, we will write $\mu\xrightarrow{E^t}\lambda$.

Let us introduce now the notion of a level that will be the main object of interest in this paper.

\begin{Def}{\rm
The {\it level} of the $m$-dimensional $n$-ary algebra $A$ is the maximal number $L$ such that there exists a sequence of non-trivial degenerations $A\xrightarrow{\not\cong}A_{L-1}\xrightarrow{\not\cong}\dots\xrightarrow{\not\cong}A_1\xrightarrow{\not\cong}A_0$ for some $m$-dimensional $n$-ary algebras $A_i$ ($0\le i\le L-1$). The level of $A$ is denoted by $lev(A)$. The {\it infinite level} of the algebra $A$ is the number defined by the equality $lev_{\infty}(A)=\lim\limits_{s\to\infty}lev(A\oplus\kk^s)$.
}\end{Def}

The aim of this paper is to classify up to isomorphism all $n$-ary algebras of the first level. For $n=2$ this was done in \cite{khud13}. We will describe also $n$-ary algebras of the first infinite level.

\begin{Def}{\rm
Let $A$ be an $n$-ary algebra, $A_0$ be a subspace of $A$ and $1\le k\le n$ be some integer.
We will call $A_0$ {\it a $k$-subalgebra} of $A$ if
\begin{itemize}
\item $[a_1,\dots,a_n]\in A_0$ for $a_1,\dots,a_n\in A$ if $|\{i\mid 1\le i\le n, a_i\in A_0\}|=k$;
\item $[a_1,\dots,a_n]=0$ for $a_1,\dots,a_n\in A$ if $|\{i\mid 1\le i\le n, a_i\in A_0\}|>k$.
\end{itemize}
}\end{Def}

Taking $k=n$ in the definition of a $k$-subalgebra, one gets the usual definition of a subalgebra of an $n$-ary algebra. On the other hand, $1$-subalgebra of $A$ is the same thing as a square zero ideal of $A$.

An important role in this paper will be played by some particular cases of degenerations.
These degenerations are generalizations of so called {\it standard In\"on\"u-Wigner contractions} (see \cite{IW}) to the case of $n$-ary algebras.
We will call them {\it IW contractions} for short. Suppose that $A_0$ is an $l$-dimensional $k$-subalgebra of $A$ for some $2\le k\le n$.
Let $\mu\in\mathcal{A}_m$ be a structure representing $A$ such that $A_0$ corresponds to the subspace $\langle e_1,\dots,e_l\rangle$ of $V$.
Then $\mu\xrightarrow{(t^{k-n}e_1,\dots,t^{k-n}e_l,t^{k-1}e_{l+1},\dots,t^{k-1}e_m)}\lambda$ for some $\lambda\in\mathcal{A}_m$ and the algebra $B$ represented by $\lambda$ will be called {\it the $k$-IW contraction of $A$ with respect to $A_0$}. The isomorphism class of the resulting algebra does not depend on the choice of the structure $\mu$ satisfying the condition stated above and always has a $k$-subalgebra $B_0\subset B$ and an $(n-k+1)$-subalgebra $B_0'\subset B$ such that $B=B_0\oplus B_0'$ as a vector space. Moreover, there exists an isomorphism $\varphi:A\rightarrow B$ of vector spaces such that $\varphi(A_0)=B_0$ and $\varphi[a_1,\dots,a_n]=[\varphi(a_1),\dots,\varphi(a_n)]$ for $a_1,\dots,a_n\in A$ if $|\{i\mid 1\le i\le n,a_i\in A_0\}|\ge k$.

\section{Some types of algebras}

In this section we introduce some special classes of algebras. As usually, by algebra we mean an $n$-ary algebra, where $n$ is some fixed integer.

Let us recall that {\it a partition} of the number $n$ is a non increasing sequence of non-negative numbers $(p_1,p_2,\dots)$ such that $\sum\limits_{i=1}^{\infty}p_i=n$. We will denote by $par_n$ the set of all partitions of the number $n$. For two partitions $p,q\in \bigcup\limits_{n\ge 1}par_n$, we say that $p$ is greater than $q$ and write $p>q$ if there is $k\ge 1$ such that $p_i=q_i$ for $1\le i\le k-1$ and $p_k>q_k$. Note that in this way we can compare any two partitions of any two positive integer numbers.
Let $Y$ be a set, $X$ be a subset of $Y$ and $y_1,\dots, y_a$ be some elements of $Y$. In this case we set $\chi_X(y_1,\dots,y_a):=|\{i\mid 1\le i\le a, y_i\in X\}|$. Let us introduce two classes of algebras.

\begin{Def}{\rm
Suppose that $p=(p_1,p_2,\dots)$ is a partition of the number $n$.
The algebra $A$ is called {\it $p$-anticommutative} if, for any chain of subspaces $A_1\subset\dots\subset A_k\subset A$ with $dim\,A_r=r$ for $1\le r\le k$, one has $[a_1,\dots,a_n]=0$ for $a_1,\dots,a_n\in A$ whenever $\chi_{A_r}(a_1,\dots,a_n)=\sum\limits_{i=1}^rp_i$ for $1\le r\le k-1$ and $\chi_{A_k}(a_1,\dots,a_n)>\sum\limits_{i=1}^kp_i$.
The algebra $A$ is called {\it strongly $p$-anticommutative} if, for any $k>0$ and any $k$-dimensional subspace $A_0\subset A$, one has $[a_1,\dots,a_n]=0$ for $a_1,\dots,a_n\in A$ whenever $\chi_{A_0}(a_1,\dots,a_n)>\sum\limits_{i=1}^kp_i$.
}
\end{Def}

\begin{Def}{\rm
Suppose that $p=(p_1,p_2,\dots)$ is a partition of the number $n-1$.
The algebra $A$ is called {\it $p$-attractive} if, for any chain of subspaces $A_1\subset\dots\subset A_k\subset A$ with $dim\,A_r=r$ for $1\le r\le k$, one has $[a_1,\dots,a_n]\in A_k$ for $a_1,\dots,a_n\in A$ whenever $\chi_{A_r}(a_1,\dots,a_n)=\sum\limits_{i=1}^rp_i$ for $1\le r\le k-1$ and $\chi_{A_k}(a_1,\dots,a_n)>\sum\limits_{i=1}^kp_i$.
The algebra $A$ is called {\it strongly $p$-attractive} if, for any $k>0$ and any $k$-dimensional subspace $A_0\subset A$, one has $[a_1,\dots,a_n]\in A_0$ for $a_1,\dots,a_n\in A$ whenever $\chi_{A_0}(a_1,\dots,a_n)>\sum\limits_{i=1}^kp_i$.
}
\end{Def}

It is clear that a strongly $p$-anticommutative algebra is $p$-anticommutative and  a strongly $p$-attractive algebra is $p$-attractive.
Note that for $p=(1,1,\dots)$ the notion a $p$-anticommutative algebra coincides with the notion of a strongly $p$-anticommutative algebra and with the notion of a usual anticommutative algebra.
For $n=2$ and $p=(1,0,\dots)$ both $p$-attractiveness and strong $p$-attractiveness of $A$ are equivalent to the fact that $ab\in\langle a,b\rangle$ for any elements $a,b\in A$; in particular, any element of $A$ is either square zero or idempotent modulo multiplication by a nonzero scalar. In terms of \cite{kpv19} this means that $A$ is of {\it generation type one}. Moreover, one can extract the classification of $(1,0,\dots)$-attractive binary algebras from the same paper. All of these algebras have level $1$ except the algebra $\kk^{m-2}\rtimes{\bf E}_4$ defined in the same paper and all binary algebras of level $1$ are $(1,0,\dots)$-attractive except the Heisenberg Lie algebra $\bf{n}_3$ and the algebra ${\bf A}_3$ from the same paper.

Note also that, for a given $p\in par_n$, $p$-anticommutative structures form a closed subset of $\mathcal{A}_m$ and, for a given $p\in par_{n-1}$, $p$-attractive structures form a closed subset of $\mathcal{A}_m$.
Indeed, let us pick some $1\le k\le m$, $\alpha_{i}^r\in\kk$ ($1\le i\le k$, $1\le r\le m$) and set $v_i=e_i$ for $1\le i\le m$ and $v_{i}:=\sum\limits_{r=1}^m\alpha_{i}^re_r$ for $m+1\le i\le m+k$. Let us also pick also some $\psi\in Map_{n,m+k}$ such that $|\psi^{-1}(m+i)|=p_i$ for $1\le i\le k-1$ and $|\psi^{-1}(m+k)|=p_k+1$. It is clear that $\mu$ represents a $p$-anticommutative algebra if $\mu(v_{\psi(1)},\dots,v_{\psi(n)})=0$ and represents a $p$-attractive algebra if $\mu(v_{\psi(1)},\dots,v_{\psi(n)})\in\langle v_{m+1},\dots,v_{m+k}\rangle$ for any choice of the described parameters. In the case of $p$-anticommutativity, it is enough to write down $\mu(v_{\psi(1)},\dots,v_{\psi(n)})$ as a linear combination of $e_1,\dots,e_m$ whose coefficients are polynomials in structure constants of $\mu$ and $\alpha_{i}^r$ ($1\le i\le k$, $1\le r\le m$). Then one has to write these polynomials as polynomials in $\alpha_{i}^r$ ($1\le i\le k$, $1\le r\le m$) with coefficients in the ring of polynomials in structure constants of $\mu$. Equating these coefficients to zero one gets polynomial equations in structure constants defining the set of $p$-anticommutative algebras. In the case of $p$-attractiveness one can note that $\mu(v_{\psi(1)},\dots,v_{\psi(n)})\in\langle v_{m+1},\dots,v_{m+k}\rangle$ for any choice of parameters if and only if $v_{m+1},\dots,v_{m+k},\mu(v_{\psi(1)},\dots,v_{\psi(n)})$ are linearly dependent for any choice of parameters. Indeed, $\mu(v_{\psi(1)},\dots,v_{\psi(n)})\in\langle v_{m+1},\dots,v_{m+k}\rangle$ clearly implies that $v_{m+1},\dots,v_{m+k},\mu(v_{\psi(1)},\dots,v_{\psi(n)})$ are linearly dependent. On the other hand, if $\mu(v_{\psi(1)},\dots,v_{\psi(n)})\not\in\langle v_{m+1},\dots,v_{m+k}\rangle$ for some choice of parameters, then choosing minimal possible $k$, one may assume that $v_{m+1},\dots,v_{m+k}$ are linearly independent, and hence $v_{m+1},\dots,v_{m+k},\mu(v_{\psi(1)},\dots,v_{\psi(n)})$ are linearly independent too. On the other hand, linear dependence is equivalent to vanishing of all $(k+1)\times (k+1)$ minors of the $(k+1)\times m$ matrix with coefficients in the ring of polynomials in structure constants of $\mu$ and $\alpha_{i}^r$ ($1\le i\le k$, $1\le r\le m$) whose $(i,r)$ ($1\le i\le k$, $1\le r\le m$) entree equals $\alpha_i^r$ and $(k+1,r)$ ($1\le r\le m$) entree equals the coefficient of $e_r$ in the decomposition of $\mu(v_{\psi(1)},\dots,v_{\psi(n)})$ by the basis $e_1,\dots,e_m$. Considering these minors as polynomials in $\alpha_{i}^r$ ($1\le i\le k$, $1\le r\le m$) over the ring of polynomials in structure constants of $\mu$ and equating all coefficients of these polynomials to zero, one gets polynomial equations in structure constants defining the set of $p$-attractive algebras.
It is not difficult to see that the defining polynomials both for $p$-anticommutativity and $p$-attractiveness are linear, and hence $p$-anticommutative and $p$-attractive algebras form two closed irreducible subvarieties of $\mathcal{A}_m$. We do not need the last mentioned fact in the current paper.

Now we introduce two classes of algebras that will play a crucial role in our classification.
Let us recall that an $n$-linear form on the space $U$ is a linear map $\Delta:{U}^{\otimes n}\rightarrow\kk$.

\begin{Def}{\rm
{\it The algebra of the $n$-linear form} $\Delta$ is the $n$-ary algebra that as a linear space is isomorphic to $U\oplus\kk$ and has the multiplication defined by the formula $[(u_1,\alpha_1),\dots,(u_n,\alpha_n)]=\big(0,\Delta(u_1\otimes\dots\otimes u_n)\big)$ for $u_1,\dots,u_n\in U$ and $\alpha_1,\dots,\alpha_n\in\kk$. Equivalently these algebras can be characterized as algebras $A$ that have a nonzero element $a\in A$ such that one has $[a_1,\dots,a_n]\in\langle a\rangle$ for any $a_1,\dots,a_n\in A$ and
$[a_1,\dots,a_n]=0$ whenever $a_i=a$ for some $1\le i\le n$.
}
\end{Def}

It is clear from the last description that the set of structures $\mu\in\mathcal{A}_m$ representing algebras of $n$-linear forms is closed; in particular, an algebra of a nonzero $n$-linear form that does 
not degenerate non-trivially to an algebra of a nonzero $n$-linear form has level one.

\begin{Def}{\rm
The $n$-ary algebra $A$ is called {\it subalgebraic} if any subspace of $A$ is a subalgebra of $A$. Equivalently, $A$ is subalgebraic if for any $a_1,\dots,a_n\in A$ one has $[a_1,\dots,a_n]\in \langle a_1,\dots,a_n\rangle$.
}
\end{Def}

Let us now explain how one can describe the set of structures representing subalgebraic algebras by polynomial equations.
If $A$ is subalgebraic, then for any $c_1,\dots,c_k\in A$ and any $\alpha_i^j\in\kk$ ($1\le i\le n$, $1\le j\le k$) one has $[a_1,\dots,a_n]\in \langle a_1,\dots,a_n\rangle\subset \langle c_1,\dots,c_k\rangle$ for $a_i=\sum\limits_{j=1}^k\alpha_i^jc_j$ ($1\le i\le n$). In particular, $c_1,\dots,c_k,[a_1,\dots,a_n]$ are linearly dependent.
Conversely, if for any $c_1,\dots,c_k\in A$ and any $\alpha_i^j\in\kk$ ($1\le i\le n$, $1\le j\le k$) the elements $c_1,\dots,c_k,[a_1,\dots,a_n]$ are linearly dependent for $a_i$ defined as above, then $A$ is subalgebraic. Really, let us pick some $a_1,\dots,a_n\in A$. Suppose that $\langle a_1,\dots,a_n\rangle$ is a $k$-dimensional space with the basis $c_1,\dots,c_k$. Then we have $a_i=\sum\limits_{j=1}^k\alpha_i^jc_j$ for some $\alpha_i^j\in\kk$ ($1\le i\le n$, $1\le j\le k$) and, by our assumption, $c_1,\dots,c_k,[a_1,\dots,a_n]$ are linearly dependent. Since $c_1,\dots,c_k$ are linearly independent, one has $[a_1,\dots,a_n]\in \langle c_1,\dots,c_k\rangle=\langle a_1,\dots,a_n\rangle$.
Now it is not difficult to write down the required equations. For each $1\le k\le n$ one takes $\beta_j^r, \alpha_i^j\in\kk$ ($1\le i\le n$, $1\le j\le k$, $1\le r\le m$) and calculate the coefficients $f^1,\dots,f^m$ in the expression
$$\mu\left(\sum\limits_{j=1}^k\alpha_1^j\sum\limits_{r=1}^m\beta_j^re_r,\dots,\sum\limits_{j=1}^k\alpha_n^j\sum\limits_{r=1}^m\beta_j^re_r\right)=\sum\limits_{r=1}^mf^re_r$$
for $\mu\in\mathcal{A}_m$.
Note that $f^r$ are polynomials in $\beta_j^r, \alpha_i^j,\mu_{x_1,\dots,x_n}^y\in\kk$ ($1\le i\le n$, $1\le j\le k$, $1\le r,x_1,\dots,x_n,y\le m$) that are linear in structure constants of $\mu$. As it is explained above, $\mu$ is subalgebraic if and only if for any choice of $k$, $\beta_j^r$ and $\alpha_i^j$ the elements $\sum\limits_{r=1}^m\beta_1^re_r,\dots,\sum\limits_{r=1}^m\beta_k^re_r,\sum\limits_{r=1}^mf^re_r$ are linearly dependent. The last condition holds if and only if all $(k+1)\times (k+1)$ minors of the matrix
$$
\begin{pmatrix}
\beta_1^1&\cdots&\beta_1^m\\
\vdots&\vdots&\vdots\\
\beta_k^1&\cdots&\beta_k^m\\
f^1&\cdots&f^m
\end{pmatrix}
$$
vanish. These minors are polynomials in $\beta_j^r$, $\alpha_i^j$ and $\mu_{x_1,\dots,x_n}^y$ linear in $\mu_{x_1,\dots,x_n}^y$. A fixed structure $\mu$ is subalgebraic if and only if, substituting the structure constants of $\mu$, one gets identically zero polynomials in $\beta_j^r$ and $\alpha_i^j$. In this way we get a set of linear equations in $\mu_{x_1,\dots,x_n}^y$ that determine the set of subalgebraic structures. In particular, this set is a closed irreducible subvariety of $\mathcal{A}_m$.

Now it is not difficult to see that structures representing algebras of $n$-linear forms and structures representing subalgebraic algebras form two closed subsets of $\mathcal{A}_m$ that intersect only at the structure with zero multiplication.
Moreover, if $A$ is not subalgebraic, then it degenerates to an algebra of a nonzero $n$-linear form. Indeed, if $A$ is not subalgebraic, then it can be represented by a structure $\mu\in\mathcal{A}_m$ such that $\mu_{i_1,\dots,i_n}^m\not=0$ for some $1\le i_1,\dots,i_n\le m-1$, and hence $\mu\xrightarrow{te_1,\dots,te_{m-1},t^ne_m}\lambda$, where $\lambda$ represents an algebras of a nonzero $n$-linear form. Thus, to classify $n$-ary algebras of the first level it is enough to consider algebras of $n$-linear forms and subalgebraic algebras. Moreover, these two classes do not intersect at the first level.

\section{Algebras of $n$-linear forms}

In this section we give the first portion of our classification. Namely, we describe algebras of $n$-linear forms with level one. From here on $Map_{a,b}$ denotes the set of all maps from $\{1,\dots,a\}$ to $\{1,\dots,b\}$.
Given $\phi\in Map_{a,b}$, $S\subset\{1,\dots,a\}$ and $1\le r\le b$, we define $\partial_S^r\phi\in Map_{a,b}$ by the equalities $\partial_S^r\phi(i)=\phi(i)$ for $i\in\{1,\dots,a\}\setminus S$ and $\partial_S^r\phi(i)=r$ for $i\in S$.

\begin{Def}{\rm
Let $p$ be a partition of $n$ with $p_m=0$ and $A$ be an $m$-dimensional algebra of an $n$-linear form.
The algebra $A$ is called {\it $p$-minimal} if it is $p$-anticommutative and can be represented by a structure $\mu\in\mathcal{A}_m$ such that if $\mu_{\phi(1),\dots,\phi(n)}^j\not=0$ for $1\le j\le m$ and $\phi\in Map_{n,m}$, then $j=m$ and $|\phi^{-1}(i)|=p_i$ for any $1\le i\le m$.
}
\end{Def}

\begin{prop}\label{nlpm} Let $p\in par_n$ be such that $p_m=0$. Suppose that the $m$-dimensional algebra  $A$ is represented by the structure $\mu$  such that $\mu(e_{\phi(1)},\dots,e_{\phi(n)})=\alpha_{\phi}e_m$ for $\phi\in Map_{n,m}$, where
$\alpha_{\phi}=0$ if $|\phi^{-1}(i)|\not=p_i$ for some $1\le i\le m-1$.
Then $A$ is $p$-minimal if and only if for any $1\le x<y\le m-1$, any $1\le k\le p_y$ and any $\psi\in Map_{n,m}$ such that $|\psi^{-1}(i)|=p_i$ for $1\le i\le m-1$, $i\not=x,y$, $|\psi^{-1}(x)|=p_x+k$ and $|\psi^{-1}(y)|=p_y-k$, one has
\begin{equation}\label{pmineq}
\sum\limits_{S\subset \psi^{-1}(x),|S|=k}\alpha_{\partial_S^y\psi}=0.
\end{equation}
\end{prop}
\begin{proof} Let us set $e_i^{\alpha}=e_i$ for $1\le i\le m$, $i\not=x$ and $e_x^{\alpha}=e_x+\alpha e_y$ with $\alpha\in \kk$ and $1\le x<y\le m$. Then, for $\psi\in Map_{n,m}$ one has
$$
\mu\big(e_{\psi(1)}^{\alpha},\dots,e_{\psi(m)}^{\alpha}\big)	=\begin{cases}
\left(\alpha^k\sum\limits_{S\subset \psi^{-1}(x),|S|=k}\alpha_{\partial_S^y\psi}\right)e_m&\mbox{ if $|\psi^{-1}(i)|=p_i$ for  $1\le i\le m$,}\\
&\mbox{ $i\not=x,y$, $|\psi^{-1}(x)|=p_x+k$,}\\
&\mbox{ $|\psi^{-1}(y)|=p_y-k$, $k\ge 0$};\\
0,&\mbox{ otherwise}.
\end{cases}
$$
Now it is clear that if $A$ is $p$-anticommutative, then the quality \eqref{pmineq} has to be satisfied for all $x$, $y$, $k$ and $\psi$ as in the statement of the proposition.

Suppose now that the quality \eqref{pmineq} is valid for all $x$, $y$, $k$ and $\psi$ as in the statement of the proposition. It is enough to show that $A$ is strongly $p$-anticommutative. Note that the condition from the definition of strong $p$-anticommutativity is satisfied if $A_0$ has a basis of the form $e_{i_1},\dots,e_{i_k}$. On the other hand, it follows from the argument above that the structure constants of $\mu$ in the basis $e_1,\dots,e_{x-1},e_x+\alpha e_y,e_{x+1},\dots,e_m$ coincide with the structure constants of $\mu$ in the basis $e_1,\dots,e_m$ for any $1\le x<y\le m$ and any $\alpha\in\kk$. Since such transformations of the basis generate all lower unitriangular transformations,  the structure constants of $\mu$ in the basis $e_1+\sum\limits_{i=2}^m\alpha_1^ie_i,e_2+\sum\limits_{i=3}^m\alpha_1^ie_i,\dots,e_m$ coincide with the structure constants of $\mu$ in the basis $e_1,\dots,e_m$ for any $\alpha_i^j\in\kk$ ($1\le i<j\le m$). Then the condition from the definition of strong $p$-anticommutativity is satisfied if $A_0$ has a basis of the form $e_{i_1}+\sum\limits_{i=i_1+1}^m\alpha_{i_1}^ie_i,\dots,e_{i_k}+\sum\limits_{i=i_k+1}^m\alpha_{i_k}^ie_i$ for some $1\le i_1<\dots<i_k\le m$. Since any $k$-dimensional space of $V$ has such a basis, $A$ is strongly $p$-anticommutative.
\end{proof}

\begin{coro}\label{nlinlev} Let $p$ be a partition of $n$ such that $p_m=0$. Then any $m$-dimensional $p$-minimal algebra of a nonzero $n$-linear form has level one.
\end{coro}
\begin{proof} Let $A$ be a $p$-minimal algebra of a nonzero $n$-linear form for some $p\in par_n$. Then it can be represented by a structure $\mu\in\mathcal{A}_m$ as in Proposition \ref{nlpm}.
Let us consider the set $\mathcal{R}$ formed by algebra structures $\mu(t)$ ($t\in\kk$) with structure constants $\mu(t)_{i_1,\dots,i_n}^j=t\mu_{i_1,\dots,i_n}^j$. Note that $\mu(t)\cong\mu$ for $t\not=0$ and $\mu(0)$ is the zero structure.
Properties of $\mu$ imply that the structure constants of $\mu(t)$ in the basis $e_1,\dots,e_{x-1},e_x+\alpha e_y,e_{x+1},\dots,e_m$ coincide with the structure constants of $\mu(t)$ in the basis $e_1,\dots,e_m$ for any $1\le x<y\le m$ and any $\alpha\in\kk$ (see the proof of Proposition \ref{nlpm}), and hence $\mathcal{R}$ is closed under lower triangular transformations of the basis $e_1,\dots,e_m$. Then $A$ degenerates only to algebras represented by structures from $\mathcal{R}$, i.e. to $A$ and to the zero algebra. Thus, $lev(A)=1$.
\end{proof}

\begin{lemma}\label{nlindeg} Any $m$-dimensional algebra of a nonzero $n$-linear form degenerates to a $p$-minimal algebra of a nonzero $n$-linear form for some $p\in par_n$.
\end{lemma}
\begin{proof} Let $A$ be an algebra of the nonzero $n$-linear form $\Delta:{U}^{\otimes n}\rightarrow\kk$. Let us consider all $q\in par_n$, $q_m=0$ such that there exist $u_1,\dots,u_n\in U$ and a basis $v_1,\dots,v_{m-1}$ of $U$ satisfying
\begin{itemize}
\item $\Delta(u_1\otimes \dots \otimes u_n)\not=0$;
\item $|\{1\le j\le n\mid u_j=v_i\}|=q_i$ for any $1\le i\le m-1$.
\end{itemize}
It is clear that the set under consideration is not empty. To see this, it is enough to choose for any basis of $U$ a nonzero product of basic elements and order the basic elements with respect to  the number of their occurrences in this product.
Let $p\in par_n$ be the maximal element of the subset of $par_n$ defined by us. It is clear from our construction that $A$ is $p$-anticommutative. Moreover, $A$ can be represented by a structure $\mu$ such that $\mu_{i_1,\dots,i_n}^j$ can be nonzero only for $1\le i_1,\dots,i_n\le m-1$ and $j=m$, and $\mu_{\phi(1),\dots,\phi(n)}^m\not=0$ for some $\phi\in Map_{n,m}$ such that $|\phi^{-1}(i)|=p_i$ for all $1\le i\le m$.

Let us choose integers $c_{m-1}>\dots>c_1>0$ such that $(1+\sum\limits_{j=i+2}^{m-1}p_j)c_{i+1}>c_i+\sum\limits_{j=i+2}^{m-1}p_jc_{j}$ for any $1\le i\le m-2$. For example, one can take $c_i=(n+1)^m-(n+1)^{m-i}$.
Let us consider the basis $E^t=\left(t^{c_1}e_1,\dots,t^{c_{m-1}}e_{m-1},t^{\sum\limits_{i=1}^{m-1}p_ic_i}e_m\right)$ for $t\in\kk^*$.
Note that for $\phi\in Map_{n,m}$ one has $\mu(e_{\phi(1)},\dots,e_{\phi(n})=t^{\sum\limits_{i=1}^{m-1}(|\phi^{-1}(i)|-p_i)c_i}\mu_{\phi(1),\dots,\phi(n)}^me_m$ and our choice of the numbers $c_1,\dots,c_{m-1}$ guarantees that $\sum\limits_{i=1}^{m-1}(|\phi^{-1}(i)|-p_i)c_i>0$ if there exist $k\ge 1$ such that $|\phi^{-1}(i)|=p_i$ for $1\le i\le k-1$ and $|\phi^{-1}(k)|<p_k$. Since $\mu$ is $p$-anticommutative, we have the degeneration $\mu\xrightarrow{E^t}\lambda$, where $\lambda$ represents a $p$-minimal  algebra of a nonzero $n$-linear form.
\end{proof}

\section{Subalgebraic algebras}

In this section we describe subalgebraic algebras of the first level giving the second half of our classification.

\begin{Def}{\rm
Let $p$ be a partition of $n-1$ with $p_{m+1}=0$ and $A$ be an $m$-dimensional $n$-ary algebra.
The algebra $A$ is called {\it maximally $p$-attractive} if it is $p$-attractive and can be represented by a structure $\mu\in\mathcal{A}_m$ such that if $\mu_{\phi(1),\dots,\phi(n)}^j\not=0$ for $1\le j\le m$ and $\phi\in Map_{n,m}$, then $|\phi^{-1}(j)|=p_j+1$ and $|\phi^{-1}(i)|=p_i$ for any $1\le i\le m$, $i\not=j$.
}
\end{Def}

Let us prove a technical lemma about $p$-attractive algebras that are similar to $p$-minimal.

\begin{lemma}\label{pattpac} Let $A$ be an $m$-dimensional $p$-attractive algebra for $p\in par_{n-1}$ with $p_{m+1}=0$. Suppose that, for some $1\le k\le m$, and some basis $b_1,\dots,b_m$ of $A$ one has
\begin{itemize}
\item $[b_{\psi(1)},\dots,b_{\psi(n)}]=0$ for $\psi\in Map_{n,m}$ such that $|\psi^{-1}(i)|<p_i$ for some $1\le i\le k-1$;
\item $[b_{\psi(1)},\dots,b_{\psi(n)}]\in\langle b_j\rangle$ for $\psi\in Map_{n,m}$ such that $|\psi^{-1}(i)|=p_i$ for $1\le i\le j-1$ and $|\psi^{-1}(j)|>p_j$ for some $1\le j\le k$.
\end{itemize}
Then $[b_{\phi(1)},\dots,b_{\phi(n)}]=0$ for $\phi\in Map_{n,m}$ such that $|\phi^{-1}(1,\dots,k)|\ge \sum\limits_{i=1}^kp_i+2$.
\end{lemma}
\begin{proof} Suppose for contradiction that $\phi\in Map_{n,m}$ is such that $|\phi^{-1}(1,\dots,k)|\ge \sum\limits_{i=1}^kp_i+2$ and $[b_{\phi(1)},\dots,b_{\phi(n)}]\not=0$.
We may assume that $|\phi^{-1}(1,\dots,k-1)|\le \sum\limits_{i=1}^{k-1}p_i+1$. Then we either have $|\phi^{-1}(i)|=p_i$ for all $1\le i\le k-1$ or  there is some $1\le j\le k-1$ such that $|\phi^{-1}(i)|=p_i$ for all $1\le i\le k-1$, $i\not=j$ and $|\phi^{-1}(j)|=p_j+1$.

Suppose first that $|\phi^{-1}(i)|=p_i$ for all $1\le i\le k-1$ and $|\phi^{-1}(k)|\ge p_{k}+2$.
Since $\sum\limits_{i=1}^{k}p_{i}<n-1$, one has $k<m$.
Let us consider the basis $b^{\alpha}=(b_1,\dots,b_{k-1},b_k+\alpha b_{k+1},b_{k+1},\dots,b_m)$. We have the expression $[b_{\phi(1)}^{\alpha}, \dots, b_{\phi(n)}^{\alpha}]=\sum\limits_{c=0}^{|\phi^{-1}(k)|}\alpha^cX_c$, where $X_c\in A$ ($0\le c\le |\phi^{-1}(k)|$) do not depend on $\alpha$. Since $X_0,X_1\in\langle b_k\rangle$ and $X_0\not=0$, we have $[b_{\phi(1)}^{\alpha}, \dots, b_{\phi(n)}^{\alpha}]\not\in \langle b_1,\dots,b_{k-1},b_{k}+\alpha b_{k+1}\rangle$ for some $\alpha\in\kk$ that contradicts the $p$-attractiveness of $A$.

Suppose now that there exists some $1\le j\le k-1$ such that $|\phi^{-1}(i)|=p_i$ for all $1\le i\le k-1$, $i\not=j$ and $|\phi^{-1}(j)|=p_j+1$. In this case $|\phi^{-1}(k)|> p_{k}$.
Let us consider the basis $b^{\alpha}=(b_1,\dots,b_{j-1},b_j+\alpha b_k,b_{j+1},\dots,b_m)$. We have
$[b_{\phi(1)}^{\alpha}, \dots, b_{\phi(n)}^{\alpha}]=[b_{\phi(1)}, \dots, b_{\phi(n)}],$
 because
$[b_{\psi(1)},\dots,b_{\psi(n)}]=0$ whenever either  $|\psi^{-1}(j)|<p_j$ or $|\psi^{-1}(i)|=p_i$ for all $1\le i\le k-1$ and $|\psi^{-1}(k)|\ge p_{k}+2$.
Since $[b_{\phi(1)}, \dots, b_{\phi(n)}]\in \langle b_{j}\rangle$ and $[b_{\phi(1)}^{\alpha}, \dots, b_{\phi(n)}^{\alpha}]\in \langle b_1,\dots,b_{j-1},b_{j}+\alpha b_{s+1}\rangle$ by $p$-attractiveness of $A$, one has $[b_{\phi(1)}, \dots, b_{\phi(n)}]=0$.
\end{proof}

\begin{prop}\label{mpa} Suppose that the $m$-dimensional algebra  $A$ is represented by the structure $\mu$  such that, for $\phi\in Map_{n,m}$, one has $\mu(e_{\phi(1)},\dots,e_{\phi(n)})=0$ if there is no $1\le j\le m$ such that $|\phi^{-1}(j)|=p_j+1$ and $|\phi^{-1}(i)|=p_i$ for all $1\le i\le m$, $i\not=j$ and $\mu(e_{\phi(1)},\dots,e_{\phi(n)})=\alpha_{\phi}e_{i_\phi}$ in the opposite case, where $|\phi^{-1}(i_\phi)|=p_{i_\phi}+1$ and $|\phi^{-1}(i)|=p_i$ for all $1\le i\le m$, $i\not=i_\phi$.
Then $A$ is maximally $p$-attractive if and only if for any $1\le x<y\le m$ the following conditions are satisfied:
\begin{itemize}
\item For any $\psi\in Map_{n,m}$ such that $|\psi^{-1}(i)|=p_i$, $i\not=x$ and $|\psi^{-1}(x)|=p_x+1$, one has
\begin{equation}\label{maxaeq1}
\sum\limits_{i\in \psi^{-1}(x)}\alpha_{\partial_{\{i\}}^y\psi}=\alpha_{\psi}.
\end{equation}
\item For any $1\le k\le p_y$ and $\psi\in Map_{n,m}$ such that $|\psi^{-1}(i)|=p_i$, $i\not=j,x,y$, $|\psi^{-1}(x)|=p_x+k+1$ and $|\psi^{-1}(y)|=p_y-k$, one has
\begin{equation}\label{maxaeq2}
\sum\limits_{S\subset \psi^{-1}(x),|S|=k}\alpha_{\partial_S^y\psi}=\sum\limits_{S\subset \psi^{-1}(x),|S|=k+1}\alpha_{\partial_S^y\psi}=0.
\end{equation}
\item For any $1\le k\le p_y$, $1\le j\le m$, $j\not=x,y$ and $\psi\in Map_{n,m}$ such that $|\psi^{-1}(j)|=p_j+1$, $|\psi^{-1}(i)|=p_i$, $i\not=j,x,y$, $|\psi^{-1}(x)|=p_x+k$ and $|\psi^{-1}(y)|=p_y-k$, one has
\begin{equation}\label{maxaeq3}
\sum\limits_{S\subset \psi^{-1}(x),|S|=k}\alpha_{\partial_S^y\psi}=0.
\end{equation}
\end{itemize}
\end{prop}
\begin{proof} Let us set $e_i^{\alpha}=e_i$ for $1\le i\le m$, $i\not=x$ and $e_x^{\alpha}=e_x+\alpha e_y$ with $\alpha\in \kk$ and $1\le x<y\le m$. Then, for $\psi\in Map_{n,m}$ one has
$$
\mu(e_{\psi(1)}^{\alpha},\dots,e_{\psi(m)}^{\alpha})=\begin{cases}
\beta_ke_x+\gamma_ke_y&\mbox{ if $|\psi^{-1}(i)|=p_i$ for all $1\le i\le m$,}\\
&\mbox{ $i\not=x,y$, $|\psi^{-1}(x)|=p_x+k+1$,}\\
&\mbox{ $|\psi^{-1}(y)|=p_y-k$, $k\ge 0$};\\
\left(\alpha^k\sum\limits_{S\subset \psi^{-1}(x),|S|=k}\alpha_{\partial_S^y\psi}\right)e_j&\mbox{ if $|\psi^{-1}(j)|=p_j+1$, $|\psi^{-1}(i)|=p_i$}\\
&\mbox{ for all $1\le i\le m$, $i\not=j,x,y$,}\\
&\mbox{ $|\psi^{-1}(x)|=p_x+k$, $|\psi^{-1}(y)|=p_y-k$,}\\
&\mbox{ $k\ge 0$, $1\le j\le m$, $j\not=x,y$};\\
0,&\mbox{ otherwise,}
\end{cases}
$$
where $\beta_k=\alpha^k\sum\limits_{S\subset \psi^{-1}(x),|S|=k}\alpha_{\partial_S^y\psi}$ and $\gamma_k=\alpha^{k+1}\sum\limits_{S\subset \psi^{-1}(x),|S|=k+1}\alpha_{\partial_S^y\psi}$.
Now it is clear that if $A$ is $p$-attractive, then due to Lemma \ref{pattpac} the equalities \eqref{maxaeq1}, \eqref{maxaeq2} and \eqref{maxaeq3} have to be satisfied for all $x$, $y$, $k$ and $\psi$ as in the statement of the proposition.

Suppose now that the equalities \eqref{maxaeq1}, \eqref{maxaeq2} and \eqref{maxaeq3} are valid for all $x$, $y$, $k$ and $\psi$ as in the statement of the proposition. It is enough to show that $A$ is strongly $p$-attractive. The proof of this fact is the same as the proof of the strong $p$-anticommutativity in Proposition \ref{nlpm}.
\end{proof}

\begin{coro}\label{subalev} Let $p$ be a partition of $n-1$ with $p_{m+1}=0$. Then any nonzero maximally $p$-attractive algebra is subalgebraic and has level one.
\end{coro}
\begin{proof} The required assertion can be deduced from Proposition \ref{mpa} in the same way as Corollary \ref{nlinlev} was deduced from Proposition \ref{nlpm}.
\end{proof}

\begin{lemma}\label{subadeg} Any nonzero $m$-dimensional subalgebraic algebra degenerates to a nonzero maximally $p$-attractive algebra for some $p\in par_{n-1}$ with $p_{m+1}=0$.
\end{lemma}
\begin{proof} It is clear that any nonzero $1$-dimensional algebra is represented by the structure $\mu$ such that $\mu(e_1,\dots,e_1)=e_1$, i.e. any $1$-dimensional algebra is maximally $(n-1,0,\dots)$-attractive. So we may assume that $m>1$.
Let $A$ be a subalgebraic algebra. Let us consider all collections $(b_1,q_1),\dots,(b_r,q_r)\in A\times\mathbb{Z}$ such that $q_1\ge\dots\ge q_r>0$, $\sum\limits_{i=1}^rq_i<n$, $b_1,\dots,b_r$ are linearly independent and there exist
 $a_1,\dots,a_n\in A$ satisfying the conditions
\begin{itemize}
\item $[a_1, \dots, a_n]\not\in \langle b_1,\dots,b_r\rangle$;
\item $|\{1\le j\le n\mid a_j=b_i\}|=q_i$ for any $1\le i\le r$.
\end{itemize}
Let us consider the set of all partitions $q=(q_1,\dots,q_r,0,\dots)\in \bigcup\limits_{i= 1}^{n-1}par_i$ arising in this way. Since $m>1$, it is clear that the partition $(1,0,\dots)\in par_1$ satisfies the conditions above, and hence the set under consideration is not empty. Let $\bar p=(\bar p_1,\dots,\bar p_r,0,\dots)\in par_{s}$ with $\bar p_r\not=0$ and $s=\sum\limits_{i=1}^r\bar p_i$ be the maximal element of this set.
By our construction, one has $[a_1,\dots,a_n]\subset \langle c_1,\dots,c_k\rangle$ for $a_1,\dots,a_n\in A$, $1\le k\le r$ and linear independent $c_1,\dots,c_k\in A$ such that
$|\{1\le j\le n\mid a_j=c_i\}|=\bar p_i$ for $1\le i\le k-1$ and $|\{1\le j\le n\mid a_j=c_k\}|>\bar p_k$.
Since $A$ is subalgebraic, it is $\bar p$-attractive whenever $s=n-1$. In this case we set $p:=\bar p$.

Suppose now that $s<n-1$. Suppose that we have $r<m-1$ at the same time. Let us pick  $a_1,\dots,a_n,b_1,\dots,b_r\in A$ such that $[a_1, \dots, a_n]\not\in \langle b_1,\dots,b_r\rangle$ and $|\{1\le j\le n\mid a_j=b_i\}|=\bar p_i$ for any $1\le i\le r$. Since $\sum\limits_{i=1}^r\bar p_i\le n-2$, we can pick two different indices $x,y\in\{1,\dots,n\}$ such that $a_x,a_y\not\in \{b_1,\dots,b_r\}$. By our construction, we have $[a_1, \dots, a_n]\in \langle b_1,\dots,b_r,a_x\rangle\cap \langle b_1,\dots,b_r,a_y\rangle$, and hence $a_x$ and $a_y$ are linearly dependent modulo $\langle b_1,\dots,b_r\rangle$. Since $r\le m-2$, there exists $a\in A$ linear independent with $a_x$ modulo $\langle b_1,\dots,b_r\rangle$. Note that
$$[a_1, \dots, a_n]=[a_1, \dots,a_{y-1}, a_y-a,a_{y+1},\dots, a_n]+[a_1, \dots,a_{y-1}, a,a_{y+1},\dots, a_n]$$
and by our construction we have
$$[a_1, \dots,a_{y-1}, a_y-a,a_{y+1},\dots, a_n]\in  \langle b_1,\dots,b_r,a_x\rangle\cap \langle b_1,\dots,b_r,a_y-a\rangle=\langle b_1,\dots,b_r\rangle$$
and
$$[a_1, \dots,a_{y-1}, a,a_{y+1},\dots, a_n]\in  \langle b_1,\dots,b_r,a_x\rangle\cap \langle b_1,\dots,b_r,a\rangle=\langle b_1,\dots,b_r\rangle.$$
This contradicts our choice of $a_1,\dots,a_n,b_1,\dots,b_r$, and hence $r=m-1$ in the case $s<n-1$.
Let us set $p_i=\bar p_i$ for $1\le i\le m-1$, $p_m=n-1-\sum\limits_{i=1}^{m-1}\bar p_i$ and $p_i=0$ for $i>m$. Let us prove that $p_m\le p_{m-1}$, and hence $p\in par_{n-1}$. Note that $A$ is automatically $p$-attractive in this case.

Suppose that $p_m>p_{m-1}$.
Let us pick  $a_1,\dots,a_n,b_1,\dots,b_{m-1}\in A$ such that $[a_1, \dots, a_n]\not\in \langle b_1,\dots,b_{m-1}\rangle$ and the cardinality of the set $I_i=\{1\le j\le n\mid a_j=b_i\}$ equals $p_i$ for any $1\le i\le m-1$.
Let us complete $b_1,\dots,b_{m-1}$ to the basis of $A$ by some element $b_m$. Then we may assume that $a_j=b_m$ for $j\in I_m=\{1,\dots,n\}\setminus\left(\bigcup\limits_{i=1}^{m-1}I_i\right)$. Then $|I_m|=p_m+1> p_{m-1}+1$.
Let us define $a_i^{\alpha}=a_i+\alpha\chi_{I_m}(i)b_{m-1}$ for $\alpha\in\kk$. We have the expression $[a_1^{\alpha}, \dots, a_n^{\alpha}]=\sum\limits_{c=0}^{p_m+1}\alpha^cX_c$, where $X_c\in A$ ($0\le c\le p_m+1$) do not depend on $\alpha$.
Note that $X_0=[a_1, \dots, a_n]\in \langle b_1,\dots,b_{m-2},b_m\rangle\setminus \langle b_1,\dots,b_{m-2}\rangle$ and $X_1\in \langle b_1,\dots,b_{m-2},b_m\rangle$ because $|I_m|-1>p_{m-1}$, and hence $[a_1^{\alpha}, \dots, a_n^{\alpha}]\not\in \langle b_1,\dots,b_{m-2},b_m+\alpha b_{m-1}\rangle$ for some $\alpha\in\kk$.
On the other hand, $[a_1^{\alpha}, \dots, a_n^{\alpha}]\in \langle b_1,\dots,b_{m-2},b_m+\alpha b_{m-1}\rangle$ for any $\alpha\in\kk$ by our assumptions. The obtained contradiction implies that $p\in par_{n-1}$.

In result, we have found $p\in par_{n-1}$ such that $A$ is $p$-attractive and represented by a structure $\mu$ such that $\mu_{\phi(1),\dots,\phi(n)}^m\not=0$ for some $\phi\in Map_{n,m}$ such that $|\phi^{-1}(i)|=p_i$ for all $1\le i\le m-1$ and $|\phi^{-1}(m)|=p_m+1$. The subspace $\langle e_1\rangle\subset V$ is a $(p_1+1)$-subalgebra of $\mu$ by Lemma \ref{pattpac}.

Let us now set $A[1]:=A$. Let us prove by induction on $s$ that there is a degeneration from $A$ to an algebra $A[s]$  ($1\le s\le m$) represented by a structure $\mu[s]$ such that
\begin{itemize}
\item $\mu[s]_{\phi(1),\dots,\phi(n)}^m\not=0$ for some $\phi\in Map_{n,m}$ such that $|\phi^{-1}(i)|=p_i$ for all $1\le i\le m-1$ and $|\phi^{-1}(m)|=p_m+1$;
\item $\langle e_1,\dots,e_r\rangle$ is a $\left(\sum\limits_{i=1}^rp_i+1\right)$-subalgebra of $\mu[s]$ for $1\le r\le s$;
\item $\langle e_{r+1},\dots,e_m\rangle$ is a $\left(\sum\limits_{i=r+1}^mp_i+1\right)$-subalgebra of $\mu[s]$ for $1\le r\le s-1$.
\end{itemize}
Suppose that we have already constructed the degeneration $A\to A[s]$ for some $1\le s\le m-1$.
Let $A[s+1]$ be the $\left(\sum\limits_{i=1}^sp_i+1\right)$-IW contraction of $A[s]$ with respect to $\langle e_1,\dots,e_s\rangle$. Let us prove that the corresponding structure $\mu[s+1]$ representing $A[s+1]$ satisfies the required conditions.
One has  $\mu[s+1]_{\phi(1),\dots,\phi(n)}^m=\mu[s]_{\phi(1),\dots,\phi(n)}^m$ for $\phi\in Map_{n,m}$ such that $|\phi^{-1}(i)|=p_i$ for all $1\le i\le m-1$ and $|\phi^{-1}(m)|=p_m+1$.
It follows from the properties of $\mu[s]$ and the $\left(\sum\limits_{i=1}^sp_i+1\right)$-IW contraction that $\langle e_1,\dots,e_r\rangle$ is a $\left(\sum\limits_{i=1}^rp_i+1\right)$-subalgebra and  $\langle e_{r+1},\dots,e_m\rangle$ is a $\left(\sum\limits_{i=r+1}^mp_i+1\right)$-subalgebra of $\mu[s+1]$ for $1\le r\le s$.

It remains to show that $\langle e_1,\dots,e_{s+1}\rangle$ is a $\left(\sum\limits_{i=1}^{s+1}p_i+1\right)$-subalgebra of $\mu[s+1]$.  The first property of a $\left(\sum\limits_{i=1}^{s+1}p_i+1\right)$-subalgebra follows directly from $p$-attractiveness of $A[s+1]$ and the fact that $\mu[s+1](e_{\phi(1)},\dots,e_{\phi(n)})=0$ for $\phi\in Map_{n,m}$ if there exists $1\le k\le s$ with $|\phi^{-1}(k)|<p_k$.

By our construction, one has $\mu[s+1](e_{\psi(1)},\dots,e_{\psi(n)})=0$ if $|\psi^{-1}(i)|<p_i$ for some $1\le i\le s$.
Moreover, since $\langle e_{j+1},\dots,e_m\rangle$ is a $\left(\sum\limits_{i=j+1}^mp_i+1\right)$-subalgebra of $\mu[s+1]$ and $A[s+1]$ is $p$-attractive, we have $\mu[s+1](e_{\psi(1)},\dots,e_{\psi(n)})\in \langle e_{j+1},\dots,e_m\rangle\cap\langle e_{1},\dots,e_{j+1}\rangle=\langle e_{j+1}\rangle$ whenever $1\le j\le s$, $|\psi^{-1}(i)|=p_i$ for all $1\le i\le j$ and $|\psi^{-1}(j+1)|> p_{j+1}$.
Then $\mu[s+1](e_{\phi(1)},\dots,e_{\phi(n)})=0$ for $\phi\in Map_{n,m}$ with $|\phi^{-1}(1,\dots,s+1)|\ge \sum\limits_{i=1}^{s+1}p_i+2$ by Lemma \ref{pattpac}.

We have constructed a degeneration from $A$ to a nonzero $p$-attractive algebra $A[m]$ represented by a structure $\mu[m]$ such that
$\langle e_1,\dots,e_r\rangle$ is a $\left(\sum\limits_{i=1}^rp_i+1\right)$-subalgebra and $\langle e_{r+1},\dots,e_m\rangle$ is a $\left(\sum\limits_{i=r+1}^mp_i+1\right)$-subalgebra of $\mu[m]$ for any $1\le r\le m-1$. These conditions guarantee that $A[m]$ is a maximally $p$-attractive algebra.
\end{proof}

\section{Main Theorem}

In this section we formulate the promised description of $n$-ary algebras of the first level and then clarify this description for $n=2,3$.
Corollaries \ref{nlinlev}, \ref{subalev}, Lemmas \ref{nlindeg}, \ref{subadeg} and the fact that any algebra of level one is either an algebra of a nonzero $n$-linear form or subalgebraic imply the next theorem.

\begin{Mtheorem} Let $A$ be an $m$-dimensional $n$-ary algebra. The algebra $A$ has level one if and only if either $A$ is a $p$-minimal algebra of a nonzero $n$-linear form for some $p\in par_n$ with $p_m=0$ or $A$ is a nonzero maximally $p$-attractive algebra for some $p\in par_{n-1}$ with $p_{m+1}=0$.
\end{Mtheorem}

\begin{rema} The description of $n$-ary algebras of level one in terms of structure constants is given in Propositions \ref{nlpm} and \ref{mpa}. To classify these algebras up to isomorphism one has to consider the systems of linear equations appearing in these propositions in more details. For example, multiplying all structure constants by a nonzero element of $\kk$ we get an isomorphic algebra and we will see that in the cases $n=2,3$ these transformations give all isomorphisms between algebra structures of level one participating in our classification, but the question if there are other isomorphisms between first level structures described here for $n>3$ we leave open in this paper.
On the other hand, it is easy to see that the partition $p$ appearing in the description of these structures is uniquely determined by the isomorphism class of an algebra.
\end{rema}

Our Main Theorem allows to describe the algebras of the first infinite level.

\begin{coro} Let $A$ be an $m$-dimensional $n$-ary algebra. The algebra $A$ has infinite level one if and only if $A$ is a $p$-minimal algebra of a nonzero $n$-linear form for some $p\in par_n$ with $p_m=0$.
\end{coro}
\begin{proof} It is clear that $A\oplus \kk$ is a $p$-minimal algebra of a nonzero $n$-linear form if $A$ is. On the other hand, $A\oplus\kk$ is	 not subalgebraic for any nonzero algebra $A$.
To see this, it is enough to choose $a_1,\dots,a_n\in A$ such that $[a_1,\dots,a_n]=\alpha a_n$ for some $\alpha\in\kk^*$ and then replace $a_i$ by $a_i+\alpha_i e$ for some $\alpha_i\in\kk$ ($1\le i\le n$), where $e$ is the basis of the direct summand $\kk$ in such a way that $a_n\not\in \langle a_1+\alpha_1e,\dots,a_n+\alpha_ne\rangle$.
Now the required assertion follows from Main Theorem.
\end{proof}

\begin{Ex} Suppose that $n=2$. Note that $|par_2|=2$ and $|par_1|=1$. Thus, we get the following classification of binary algebras of level one:
\begin{enumerate}
\item If $p=(2,0,\dots)$, then the only $p$-minimal algebra of a nonzero bilinear form is represented by the structure $\bf{A}_3$ with the only nonzero product $e_1e_1=e_2$.
\item If $p=(1,1,0,\dots)$, then the only $p$-minimal algebra of a nonzero bilinear form is represented by the structure $\bf{n}_3$ with nonzero products $e_1e_2=e_3$ and $e_2e_1=-e_3$.
\item If $p=(1,0,\dots)$, then any nonzero maximally $p$-attractive algebra is represented either by the structure $p^-$ with nonzero products $e_1e_i=e_i$ and $e_ie_1=-e_i$ for $2\le i\le m$ or by a structure $\nu^{\alpha}$ ($\alpha\in\kk$) with nonzero products $e_1e_1=e_1$, $e_1e_i=\alpha e_i$ and $e_ie_1=(1-\alpha)e_i$ for $2\le i\le m$.
\end{enumerate}
This classification agrees with the results of \cite{kpv17}.
\end{Ex}

\begin{Ex} Suppose that $n=3$. Note that $par_3=\{(3,0,\dots),(2,1,0,\dots),(1,1,1,0,\dots)\}$ and $par_2=\{(2,0,\dots),(1,1,0,\dots)\}$. Thus, using Propositions \ref{nlpm} and \ref{mpa}, we get the following classification of ternary algebras of level one:
\begin{enumerate}
\item If $p=(3,0,\dots)$, then the only $p$-minimal algebra of nonzero $3$-linear form is represented by the structure with the only nonzero product $[e_1,e_1,e_1]=e_2$.
\item If $p=(2,1,0,\dots)$, then any  $p$-minimal algebra of nonzero $3$-linear form is represented by a structure with nonzero products $[e_1,e_1,e_2]=\alpha_3 e_3$, $[e_1,e_2,e_1]=\alpha_2 e_3$ and $[e_2,e_1,e_1]=\alpha_1 e_3$, where $\alpha_1,\alpha_2,\alpha_3\in\kk$ are not all zero and $\alpha_1+\alpha_2+\alpha_3=0$. The triple $(\alpha_1,\alpha_2,\alpha_3)$ is determined by the isomorphism class of the algebra up to multiplication by a nonzero element of $\kk$.
\item If $p=(1,1,1,0,\dots)$, then the only $p$-minimal algebra of nonzero $3$-linear form is represented by the structure with nonzero products $[e_{\sigma(1)},e_{\sigma(2)},e_{\sigma(3)}]=sign(\sigma) e_4$ for all permutations $\sigma$ of the set $\{1,2,3\}$, where $sign(\sigma)$ is the sign of the permutation $\sigma$.
\item If $p=(2,0,\dots)$, then any nonzero maximally $p$-attractive algebra is represented by the structure with nonzero products $[e_1,e_1,e_1]=\epsilon e_1$, $[e_1,e_1,e_i]=\alpha_3 e_i$, $[e_1,e_i,e_1]=\alpha_2 e_i$ and $[e_i,e_1,e_1]=\alpha_1 e_i$  for $2\le i\le m$, where $\alpha_1,\alpha_2,\alpha_3\in\kk$, $\epsilon\in\{0,1\}$ and $\alpha_1+\alpha_2+\alpha_3=\epsilon$. If $\epsilon=0$, then the triple $(\alpha_1,\alpha_2,\alpha_3)$ is nonzero and is determined by the isomorphism class of the algebra up to multiplication by a nonzero element of $\kk$. If $\epsilon=1$, then the triple $(\alpha_1,\alpha_2,\alpha_3)$ is uniquely determined by the isomorphism class of the algebra.
\item If $p=(1,1,0,\dots)$, then any nonzero maximally $p$-attractive algebra is represented by the structure with nonzero products $[e_1,e_1,e_2]=(\alpha_2-\alpha_3)e_1$, $[e_1,e_2,e_1]=(\alpha_3-\alpha_1)e_1$, $[e_2,e_1,e_1]=(\alpha_1-\alpha_2)e_1$, $[e_2,e_2,e_1]=(\alpha_3-\alpha_2)e_2$, $[e_2,e_1,e_2]=(\alpha_1-\alpha_3)e_2$, $[e_1,e_2,e_2]=(\alpha_2-\alpha_1)e_2$, $[e_i,e_1,e_2]=\alpha_1e_i$, $[e_i,e_2,e_1]=-\alpha_1e_i$, $[e_1,e_i,e_2]=-\alpha_2e_i$, $[e_2,e_i,e_1]=\alpha_2e_i$, $[e_1,e_2,e_i]=\alpha_3e_i$, $[e_2,e_1,e_i]=-\alpha_3e_i$ for $3\le i\le m$, where $\alpha_1,\alpha_2,\alpha_3\in\kk$ are not all zero and $\alpha_1+\alpha_2+\alpha_3=0$ if $m>2$. The triple $(\alpha_1,\alpha_2,\alpha_3)$ is determined by the isomorphism class of the algebra up to multiplication by a nonzero element of $\kk$ if $m>2$. If $m=2$, then not all $\alpha_1,\alpha_2,\alpha_3$ are equal and the triple $(\alpha_1,\alpha_2,\alpha_3)$ is determined by the isomorphism class of the algebra up to multiplication by a nonzero element of $\kk$ and addition of an element of $\kk$, but the last mentioned transformation does not change the structure representing the algebra.
\end{enumerate}
\end{Ex}

\bigskip

{\bf Acknowledgements.} The work was supported by the Russian Science Foundation research project number 19-71-10016. The author is a Young Russian Mathematics award winner and would like to thank its sponsors and jury.

\noindent{{\bf Addresses:}
\newline
Yury Volkov \\
Saint-Petersburg State University\\
Universitetskaya nab. 7-9, St. Peterburg, Russia\\
e-mail:  wolf86\_666@list.ru}
\end{document}